\documentclass{amsart}

\usepackage{amssymb}
\usepackage{latexsym}
\usepackage{euscript}
\usepackage{mathtools}

\usepackage[usenames,dvipsnames,svgnames,table]{xcolor}

   \def\sH{{\mathfrak H}}   
      \def\sL{{\mathfrak L}}

\def\dC{{\mathbb C}}

   \def\cH{{\mathcal H}}

\def\Lin{{\sL}}
\def\gl{{\lambda}}
\newcommand*{\ceil} [1]{\lceil#1\rceil}

\newcommand*{\R}{{\mathbb{R}}}     
\newcommand*{\N}{{\mathbb{N}}}     
\newcommand{\eq}[1]{\begin{align*}#1\end{align*}}
\newcommand{\eqn}[1]{\begin{align}#1\end{align}}
\newcommand*{\abs} [1]{\lvert#1\rvert}
\newcommand*{\set} [1]{\{#1\}}
\newcommand*{\iprod}[2]{\langle#1,#2\rangle}    
\DeclareMathOperator{\re}{Re}			

\def\t#1{{{\tilde #1} }}
\def\wt#1{{{\widetilde #1} }}

\def\bm\chi{\mbox{\boldmath$\chi$}}

\def\ker{{\rm ker\,}}
\def\ran{{\rm ran\,}}

\def\dom{{\rm dom\,}}
\def\Dom{{\rm dom\,}}

\def\dist{{\rm dist\,}}

\let\xker=\ker \def\ker{{\xker\,}}

\def\spann{{\rm span}}

\def\sign{{\rm sign\,}}

\newcommand{\pmat}[1]{\begin{pmatrix}#1\end{pmatrix}}
\newcommand*{\norm}[1]{\lVert#1\rVert}
\renewcommand {\l}{\lambda}
\renewcommand {\a}{\alpha}
\renewcommand {\b}{\beta}

\renewcommand {\t}{\theta}
\newcommand {\e}{\varepsilon}

\renewcommand {\L}{\Lambda}
\newcommand {\ov}{\overline}
\newcommand {\m}{\mu}

\newcommand {\gb}{\beta}
\renewcommand {\gg}{\gamma}

\newcommand{\WPgamma}{\eta}
\newcommand{\WPgammaRan}{(0,\WPgamma_0]}
\newcommand{\deltaRan}{(0,\delta_0]}

\newcommand{\Mpert}{M_T}

\newtheorem{theorem}{Theorem}[section]
\newtheorem{proposition}[theorem]{Proposition}
\newtheorem{corollary}[theorem]{Corollary}
\newtheorem{lemma}[theorem]{Lemma}

\theoremstyle{definition}
\newtheorem{example}[theorem]{Example}
\newtheorem{remark}[theorem]{Remark}
\newtheorem{definition}[theorem]{Definition}

\numberwithin{equation}{section}

\begin{document}

\title{Robustness of polynomial stability of damped wave equations}
\author{Dmytro Baidiuk}
\address[D. Baidiuk]{Mathematics and Statistics, Faculty of Information Technology and Communication Sciences, Tampere University, P.O. Box 692, 33101 Tampere, Finland}
\email{baydyuk@gmail.com}

\author{Lassi Paunonen}
\address[L. Paunonen]{Mathematics and Statistics, Faculty of Information Technology and Communication Sciences, Tampere University, P.O. Box 692, 33101 Tampere, Finland}
\thanks{The research was funded by the Academy of Finland grants 298182 and 310489 held by L. Paunonen.}

\email{lassi.paunonen@tuni.fi}

\begin{abstract}
In this paper we present new results on the preservation of polynomial stability of damped wave equations under addition of perturbing terms. We in particular introduce sufficient conditions for the stability of perturbed  two-dimensional wave equations on rectangular domains, a one-dimensional weakly damped Webster's equation, and a wave equation with an acoustic boundary condition. In the case of Webster's equation, we use our results to compute explicit numerical bounds that guarantee the polynomial stability of the perturbed equation.
\end{abstract}

\maketitle

\section{Introduction}

In this paper we study the stability properties of
damped wave equations and 
abstract second-order differential equations of the form~\cite{EngNag00book,TucWei09book}
\begin{equation}\label{GWE}
  \left\{
  \begin{array}{ll}
    w_{tt}(t)-L w(t)+D_0D_0^*w_t(t)=0,    & \qquad t>0,\\
    w(0)=w_0,\quad
    w_t(0)=w_1 
  \end{array}\right.
\end{equation}
on a Hilbert space $X_0$.
Here $L: \Dom L \subset X_0\to X_0$ is a negative self-adjoint operator with a bounded inverse and
$D_0\in \sL(U,X_0)$ for some Hilbert space $U$.
Our main interest is in the preservation of stability under bounded perturbations in the situation where the unperturbed differential equation~\eqref{GWE} is only \emph{polynomially stable}~\cite{BatEng06,BorTom10} (as opposed to being uniformly exponentially stable). The polynomial stability of~\eqref{GWE} means that there exist constants $\alpha,M>0$ such that
for all initial conditions $w_0\in \Dom L $ and $w_1\in \Dom (-L)^{1/2}$ the solutions of~\eqref{GWE} satisfy~\cite{BorTom10}
\begin{align}
  \label{GWEpolstab}
\norm{(-L)^{1/2}w(t)}^2 + \norm{w_t(t)}^2\leq \frac{M}{t^{2/\alpha}}
\left(
  \norm{Lw_0}^2 + \norm{(-L)^{1/2}w_1}^2
\right),
\quad t>0.
\end{align}
Polynomial stability has been investigated in detail in the literature for
damped wave equations on multi-dimensional domains~\cite{LiuRao05,BurHit07,AnaLea14}, coupled partial differential equations~\cite{ZhaZua07,Duy07,AvaTri13,RivAvi15}, as well as abstract damped second-order systems of the form~\eqref{GWE}~\cite{AmmTuc01,LiuZha15,AmmBch17,DelPat21,ChiPau19arxiv}.

Polynomial stability is a strictly weaker concept than exponential stability, and it can in particular be destroyed under addition of arbitrarily small lower order terms in the partial differential equation.
In this paper we employ and refine the general framework introduced in~\cite{Pau12,Pau14c} to
present conditions for preservation of the polynomial stability of the abstract differential equation~\eqref{GWE} under finite-rank and Hilbert--Schmidt perturbations.
Moreover, we study preservation of polynomial stability for
selected partial differential equation models, namely, a damped two-dimensional wave equation on a rectangular domain, a weakly damped Webster's equation, and a one-dimensional wave equation with a dynamic boundary condition.

As our first main results we present
general conditions for the polynomial stability
of perturbed second-order systems of the form
\begin{equation}\label{PGWE}
  \left\{
  \begin{array}{ll}
    w_{tt}(t)-L w(t)+D_0D_0^*w_t(t)
    =B_2(C_1w(t)+ C_2w_t(t)),  & \quad t>0\\
    w(0)=w_0,\quad
    w_t(0)=w_1 .
  \end{array}\right.
\end{equation}
Here the operators $B_2\in \sL(Y,X_0)$, $C_1\in\sL(\dom (-L)^{1/2},Y)$, and $C_2\in \sL(X_0,Y)$ for some Hilbert space $Y$ describe the perturbations to the nominal polynomially stable equation~\eqref{GWE}.
As our first main results we adapt and improve the main results in~\cite{Pau12,Pau14c} to make them more easily verifiable for second-order systems of the form~\eqref{PGWE}.
Our results show that if the unperturbed equation is polynomially stable so that~\eqref{GWEpolstab} is satisfied with some $\alpha\in(0,2]$, then~\eqref{PGWE} is polynomially stable provided that
for exponents $\beta,\gamma\in[0,1]$ satisfying $\beta+\gamma\geq \alpha$ the graph norms $\norm{(-L)^{\beta/2}B_2}$, $\norm{(-L)^{(\gamma-1)/2}C_1^\ast}$, and $\norm{(-L)^{\gamma/2}C_2^\ast}$ are finite and sufficiently small.
  Our new results also provide 
  concrete bounds for the required sizes of these graph norms based 
  on lower bounds for the operator $D_0^\ast $ restricted to the spectral subspaces of $L$.
The results are applicable in the situations where $Y$ is either finite-dimensional or where $B_2$, $C_1$, and $C_2$ are Hilbert--Schmidt operators.

As the first concrete partial differential equation we study
a wave equation with viscous damping on a rectangle $\Omega=(0,a)\times (0,b)$,
\begin{equation}
  \label{E:DWE}
  \left\{  \begin{array}{ll}
    w_{tt}(t,x,y)-\Delta w(t,x,y)+d(x,y)w_t(t,x,y)=0,   \quad \qquad&t>0, (x,y)\in \Omega\\
    w(t,x,y)=0, \hspace{3cm} & t>0,(x,y)\in\partial \Omega.
  \end{array}
  \right.
\end{equation}
We assume the damping coefficient $d(\cdot,\cdot)\geq 0$ is strictly positive on some non-empty open subset of $\Omega$ which does not satisfy the Geometric Control Condition (see, e.g., \cite{AnaLea14}).
We apply our abstract results to present conditions for the polynomial stability of perturbed wave equations of the form
    \begin{equation*}
    \begin{split}
      &w_{tt}(t,x,y)-\Delta w(t,x,y)+d(x,y)w_t(t,x,y)\\
      &\quad =\sum_{k=1}^m b_{k,2}(x,y)\int\limits_\Omega (w(t,\xi,\eta) c_{k,1}(\xi,\eta)+ w_t(t,\xi,\eta)c_{k,2}(\xi,\eta))d\xi d\eta
      \end{split}
    \end{equation*}
    where $b_{k,2},c_{k,1},c_{k,2}\in L^2(\Omega)$.
    In particular, our results show that the perturbed wave equation is polynomially stable provided that the coefficient functions $b_{k,2}$, $c_{k,1}$, and $c_{k,2}$ have sufficient \emph{smoothness properties} in the sense that these functions belong to fractional domains of $-\Delta$ (the Dirichlet Laplacian on $\Omega$), and the associated fractional graph norms
    are sufficiently small.
We present also analogous results of Hilbert--Schmidt perturbations of~\eqref{E:DWE}.
Finally, we analyse the stability of~\eqref{E:DWE} in a situation where the damping term is perturbed in a non-dissipative way with a rank one operator.

Our second concrete partial differential equation is a Webster's equation with a weak damping on $(0,1)$,
\begin{equation*}
  \left\{
  \begin{array}{ll}
    w_{tt}(t,x)
    - w_{xx}(t,x)-aw_x(t,x)
    +d(x)\int_0^1w_t(t,\xi)d(\xi)e^{a\xi} d\xi =0    \\
    w(t,0)=w(t,1)=0  \\
    w(0,x)=w_0(x),\quad w_t(0,x)=w_1(x),
  \end{array}\right.
\end{equation*}
 where $a>0$ and $d(\cdot)\in L^2(0,1)$ is the damping coefficient. In this article we focus on a 
  special case where $d(x)=1-x$. 
 We begin by proving that the Webster's equation is polynomially stable with this particular damping coefficient.
 Using our abstract results we then present conditions for the preservation of the Webster's equation under addition of a perturbation term.
  We also present a numerical example where we 
  compute numerical bounds for the coefficient functions in the perturbation to guarantee the preservation of polynomial stability
  of the Webster's equation.

Finally, as our third partial differential equation we consider a one-dimensional wave equation with a dynamic boundary condition. 
The polynomial stability of this model was shown in~\cite{RivQin03,AbbNic13}, 
and in this paper we present conditions for the preservation of the stability under addition of perturbation terms to the differential equation.

We use the following notation. Given a closed operator $A$ on a Hilbert space $X$, which will be assumed to be complex, we denote its domain by $\dom A$, its kernel by $\ker A$, and its range by $\ran A$. The spectrum of $A$ is denoted by $\sigma(A)$, and given $\l\in \rho(A):=\mathbb C\setminus\sigma (A)$ we write $R(\l,A)$ for the resolvent operator $(\l-A)^{-1}$. The space of bounded linear operators on $X$ is denoted by $\sL(X)$.  
Given two functions $f,g:(0,\infty)\to\mathbb R_+$, we write $f(t)=O(g(t))$ to indicate that $f(t)\leq Cg(t)$ for some constant $C>0$ and for all sufficiently large $t>0$.

\section{Robustness of stability for generalized wave equations}\label{S:GenWE}
\begin{subsection}{Polynomial stability of strongly continuous semigroups}
  \label{sec:PolStab}

The second-order differential equation~\eqref{GWE}
with a negative and boundedly invertible operator  $L: \Dom L \subset X_0\to X_0$  and 
$D_0\in \sL(U,X_0)$ 
 can be represented as a first-order abstract Cauchy problem with state
$u(t)=
(  w(t),
  w_t(t)
)^\top$
as
\begin{equation*}
 \frac{d u}{d t}=Au,\quad \text{ where }\quad
A=\begin{pmatrix}
  0&I\\
  L & -D_0D_0^*
\end{pmatrix}
\end{equation*}
with the initial condition $u(0)=(w_0,w_1)^\top$.
We choose the state space of this linear system as
$$
\cH=\dom (-L)^{1/2}\times X_0.
$$
The space $\cH$  is a Hilbert space with inner product defined by
$$
\langle u,v\rangle_\cH=\left\langle (-L)^{1/2}u_1,(-L)^{1/2}v_1\right\rangle_{X_0}+\langle u_2,v_2\rangle_{X_0}
$$
for all $u=(u_1,u_2)^\top,v=(v_1,v_2)^\top\in\cH$.
The domain of $A$ is $\dom A=\dom L\times \dom (-L)^{1/2}$.
The operator $A$ has the form $A=A_0-DD^\ast$ where
\eqn{
  \label{eq:A0D0}
  A_0
=\begin{pmatrix}
  0&I\\
  L & 0
\end{pmatrix}: \Dom A \subset \cH\to \cH \qquad \mbox{and} \qquad D = 
\begin{pmatrix}
  0\\D_0
\end{pmatrix}\in \Lin(U,\cH).
}
Here $A_0$ is a skew-adjoint operator and $A$  generates a strongly continuous semigroup $T(t)$ on $\cH$ by the Lumer-Phillips theorem~\cite[Sec.~VI.3]{EngNag00book}.

\begin{definition}[{\cite{BorTom10}}]
  \label{PolSt}
 A strongly continuous semigroup $T(t)$ generated by a linear operator $A$ is said to be \emph{polynomially stable} with $\a>0$ if it is uniformly bounded, i.e. $\sup_{t\geq0}\|T(t)\|<\infty$, if $i\mathbb R\subset\rho(A)$, and if
 \begin{equation*}
 \left\|T(t)A^{-1}\right\|\leq \frac{M}{t^{1/\alpha}},\quad \text{for all }t>0
 \end{equation*}
 for some constant $M>0$.
 \end{definition}

 \subsection{Polynomial stability of perturbed semigroups}

 We are interested in robustness of the polynomial stability of~\eqref{GWE} under perturbations of the form
 \begin{equation}\label{PGWE2}
  \left\{
  \begin{array}{ll}
    w_{tt}(t)-L w(t)+D_0D_0^*w_t(t)
    =B_2(C_1w(t)+ C_2w_t(t)),  & \quad t>0\\
    w(0)=w_0,\quad
    w_t(0)=w_1 & \\
  \end{array}\right.
\end{equation}
where $B_2\in \Lin(Y,X_0)$,
$C_1\colon \Dom  (-L)^{1/2}\subset X_0\to Y$,
and $C_2\in \Lin(X_0, Y)$ for some Hilbert space $Y$ are such that  $C_1(-L)^{-1/2}\in \Lin(X_0,Y)$.
If we define $B:=(0,B_2)^\top\in\sL(Y,\cH)$ and  $C:=(C_1,C_2)\in\sL(\cH,Y)$,
 the perturbed system can be represented as an abstract Cauchy problem $\frac{d u}{d t}=(A+BC)u$.
The following theorem presented in~\cite{Pau14c} provides general conditions for the preservation of the polynomial stability of the semigroup $T_{A+BC}(t)$ generated by $A+BC$.
\begin{theorem}[{\cite[Thm.~6]{Pau14c}}]
  \label{T:1.1}
   Assume $T(t)$ generated by $A$ is polynomially stable with $\alpha>0$, let $\beta,\gamma\geq 0$ be such that $\beta+\gamma\geq\alpha$, 
   and let $\kappa>0$
      satisfy 
      $$
      \kappa<  \frac{1}{\sup_{\lambda\in\dC_+}\norm{R(\lambda,A)(-A)^{-\beta-\gamma}}^{1/2}}.
      $$
If $B\in\sL(Y,\cH)$ and $C\in\sL(\cH,Y)$ are such that
   \begin{equation}\label{E:1.2}
     \ran B\subset\dom(-A)^\beta, \quad \ran C^*\subset\dom(-A^*)^\gamma,
   \end{equation}
   if $(-A)^\beta B$ and $(-A^*)^\gamma C^*$ are Hilbert--Schmidt operators, and if
   \begin{equation}\label{E:L}
   \left\|(-A)^\beta B\right\|<\kappa,\quad\left\|(-A^*)^\gamma C^*\right\|<\kappa,
    \end{equation}
    then the semigroup generated by $A+BC$ is polynomially stable with the same $\alpha$.
 \end{theorem}

The following theorem introduces a concrete bound $\kappa>0$ for the norms of the perturbations in Theorem~\ref{T:1.1}
for $A=A_0-DD^\ast$ with a skew-adjoint operator $A_0$
in the important special case
$\gb,\gg\geq 0$ are chosen so that $\gb+\gg=\ceil{\alpha}$ (here $\ceil{\alpha}\in\N$ denotes 
the ceiling of $\alpha>0$).
The first part of the result is a special case of~\cite[Thm.~3.5]{ChiPau19arxiv} with a proof which has been modified in a trivial manner to yield an explicit constant~$M_R>0$.

\begin{theorem}
  \label{thm:kappaestimate}
  Let $X$ and $U$ be Hilbert spaces, and
  assume $A=A_0-DD^\ast$ where $A_0: \Dom A_0 \subset X\to X$ is skew-adjoint and $D\in \Lin(U,X)$.
  Let $P_{(a,b)}\in \Lin(X)$ be the spectral projection of $A_0$ corresponding to the interval $(ia,ib)\subset i\mathbb{R}$.
  Assume there exist $\WPgamma_0,\delta_0>0$ and functions $\WPgamma :\R\to(0,\WPgamma_0] $ and $\delta :\mathbb{R}\to (0,\delta_0]$ such that
  \begin{align}
    \label{eq:WPestimate}
    \norm{D^\ast x}\geq \WPgamma(s)\norm{x}, \qquad \forall x\in\ran P_{(s-\delta(s),s+\delta(s))}.
  \end{align}
  Then 
  \begin{align*}
    \norm{R(is,A)}\leq \frac{M_R}{\WPgamma(s)^{2}\delta(s)^{2}}, \qquad \forall s\in\mathbb{R},
  \end{align*}
  where
  \eq{
    M_R = 2\sqrt{\WPgamma_0^4\delta_0^2 +2\WPgamma_0^2\delta_0^2\norm{D}^2 
+(\delta_0^2  + \WPgamma_0^2 \norm{D}^2+2\norm{D}^4 )^2}.
  }
If there exists $M_0>0$ such that $\WPgamma(s)^{-2}\delta(s)^{-2}\leq M_0(1+|s|^{\alpha}) $ for all $s\in\R$, then for $\gb,\gg\geq 0$ with $\gb+\gg=\ceil{\alpha}$ 
in Theorem~\textup{\ref{T:1.1}}
it is possible to choose any $\kappa>0$ such that  
\begin{align*}
  \kappa &< \frac{1}{ \sqrt{2M_C}}, 
    \end{align*}
    where $M_C>0$ 
    is defined with an arbitrary $s_0>0$ by
  \begin{align*}
  M_C &=
  \max \left\{
    M_RM_0\norm{A^{-1}}^{\ceil{\alpha}} (1+s_0^\alpha),
\frac{M_RM_0(1+s_0^{\alpha})}{s_0^{\ceil{\alpha}}}
+ \sum_{k=1}^{\ceil{\alpha}}
\frac{\norm{A^{-1}}^k}{
  s_0^{\ceil{\alpha}+1-k} }
  \right\}.
\end{align*}
\end{theorem}

\begin{proof}
  Assume that the functions $\WPgamma$ and $\delta$ satisfy the assumptions of the theorem. Let $y\in X$ and $s\in\R$ be arbitrary and write $x=R(is,A)y\in \Dom A $. We then have $(is-A_0+DD^\ast)x=y$, and thus
  \eq{
    \norm{D^\ast x}^2 
    = \re \iprod{DD^\ast x}{x}
    = \re \iprod{(is-A_0+DD^\ast) x}{x}
    = \re \iprod{y}{x}
    \leq \norm{y}\norm{x}.
  }
Denote $P_s:=P_{(s-\delta(s),s+\delta(s))}$ for brevity and 
write $X=X_s\oplus_\perp X_\infty$ where $X_s=P_sX$ and $X_\infty = (I-P_s)X$. 
If we write $x=x_0+x_\infty$ and $y=y_0+y_\infty$ according to this decomposition, then
\eq{
  &(is-A_0)x_\infty +(I-P_s)DD^\ast x=y_\infty, \\
  \Leftrightarrow \quad & x_\infty 
  =(is-A_0)^{-1}\left[ y_\infty - (I-P_s)DD^\ast x\right],
}
since the restriction  $(is-A_0)\vert_{X_\infty}$ of $is-A_0$ to $X_\infty$ is boundedly invertible. Since $A_0$ is skew-adjoint and $\sigma((is-A_0)\vert_{X_\infty})\subset i\R\setminus (-i\delta(s),i\delta(s))$,
we have 
\eq{
  \norm{x_\infty}^2
&\leq \delta(s)^{-2} \norm{y_\infty - (I-P_s)DD^\ast x}^2\\
&\leq 2\delta(s)^{-2} \left(\norm{y}^2 + \norm{D}^2\norm{D^\ast x}^2\right).
}
By assumption we have $\norm{x_0}\leq \WPgamma(s)^{-1}\norm{D^\ast x_0}\leq \WPgamma(s)^{-1}(\norm{D^\ast x}+\norm{D^\ast x_\infty})$. If we denote $q(s)= 1+ 2\WPgamma(s)^{-2}\norm{D}^2$, we can use $\norm{D^\ast x}^2\leq \norm{x}\norm{y}$ and the Young's inequality to estimate
\eq{
  \norm{x}^2 
  &= \norm{x_0}^2 + \norm{x_\infty}^2 
  \leq 2\WPgamma(s)^{-2}(\norm{D^\ast x}^2 + \norm{D}^2 \norm{x_\infty}^2) + \norm{x_\infty}^2 \\
  &= 2\WPgamma(s)^{-2} \norm{D^\ast x}^2
  +q(s)  \norm{x_\infty}^2 \\
  &\leq 2\WPgamma(s)^{-2} \norm{D^\ast x}^2 +2\delta(s)^{-2} q(s)  \left(\norm{y}^2 + \norm{D}^2\norm{D^\ast x}^2\right)\\
  &\leq 2\delta(s)^{-2} q(s)  \norm{y}^2+
  2(\WPgamma(s)^{-2}  +\delta(s)^{-2} q(s)   \norm{D}^2) \norm{x}\norm{y}\\
  &\leq 2\delta(s)^{-2} q(s)  \norm{y}^2+\frac{1}{2}\norm{x}^2+
  2(\WPgamma(s)^{-2}  +\delta(s)^{-2} q(s)   \norm{D}^2)^2 \norm{y}^2.
}
This estimate implies
\eqn{
  \label{eq:kappaResxybound}
  \norm{x}^2\leq 4(\delta(s)^{-2} q(s)+(\WPgamma(s)^{-2}  +\delta(s)^{-2} q(s)   \norm{D}^2)^2)\norm{y}^2.
}
Recall that 
$1\leq \WPgamma_0^2\WPgamma(s)^{-2} $ and $1\leq \delta_0^2\delta(s)^{-2} $.
We have
\eq{
  q(s)
  &=1+2\WPgamma(s)^{-2}\norm{D}^2 
  \leq \WPgamma(s)^{-2}(\WPgamma_0^2+2\norm{D}^2 ) .
}
The estimate~\eqref{eq:kappaResxybound} implies
  \eq{
    \norm{R(is,A)}^2
    &\leq  4(\delta(s)^{-2} q(s)+(\WPgamma(s)^{-2}  +\delta(s)^{-2} q(s)   \norm{D}^2)^2)\\
    &\leq  4(\WPgamma_0^2\delta_0^2 (\WPgamma_0^2+2\norm{D}^2 )
    +(\delta_0^2  + (\WPgamma_0^2+2\norm{D}^2 )   \norm{D}^2)^2)\WPgamma(s)^{-4}\delta(s)^{-4}.
  }
This completes the first part of the proof.

Assume now that there exists $M_0>0$ such that $\WPgamma(s)^{-2}\delta(s)^{-2}\leq M_0(1+|s|^\alpha) $ for all $s\in\R$ and denote $n_\alpha = \ceil{\alpha}\in\N$. Then $\norm{R(is,A)}\leq M_0M_R(1+|s|^\alpha)$ for all $s\in\R$, and Theorem~\ref{T:1.1} implies that if $\gb,\gg\geq 0$ are such that
$\gb+\gg=n_\alpha$, then the constant $\kappa>0$ is required to satisfy $\kappa<(\sup_{\gl\in\mathbb{C}_+}\norm{R(\gl,A)A^{-n_\alpha}})^{-1/2}$.
The approach in the proof of~\cite[Lem.~5.3]{BatChi16} can be used to show that 
\eq{
 \sup_{\lambda\in\mathbb{C}_+}\norm{R(\lambda,A)A^{-n_\alpha}} \leq 2\cdot\sup_{s\in\mathbb{R}}\;\norm{R(is,A)A^{-n_\alpha}}.
}
This implies that $\kappa>0$ in Theorem~\ref{T:1.1} can be chosen 
to have any value $\kappa<1/ \sqrt{2M_C}$ provided that the constant $M_C>0$ in the statement of the theorem is such that $\norm{R(is,A)A^{-n_\alpha}}\leq M_C$ for all $s\in\R$.
In order to show this, let $s_0>0$ be arbitrary and fixed. For any $s\in\R$ with $\abs{s}\leq s_0$ we have
\eq{
  \norm{R(is,A)A^{-n_\alpha}}\leq M_RM_0\norm{A^{-1}}^{n_\alpha} (1+s_0^\alpha).
}
On the other hand, if $\abs{s}\geq s_0$, then
using the resolvent identity $R(is,A)A^{-1} = (is)^{-1} (R(is,A)+A^{-1})$ 
repeatedly shows that
\eq{
  \norm{R(is,A)A^{-n_\alpha}} 
& =\norm{ (is)^{-n_\alpha} R(is,A) + \sum_{k=1}^{n_\alpha}(is)^{k-1-n_\alpha}A^{-k}} \\
&\leq \frac{M_RM_0(1+|s|^{\alpha})}{|s|^{n_\alpha}}
+ \sum_{k=1}^{n_\alpha}|s|^{k-1-n_\alpha}\norm{A^{-1}}^k 
 \\
&\leq
\frac{M_RM_0(1+s_0^{\alpha})}{s_0^{n_\alpha}}
+ \sum_{k=1}^{n_\alpha}s_0^{k-1-n_\alpha}\norm{A^{-1}}^k .
}
Combining the above two estimates shows that $\sup_{s\in\R}\norm{R(is,A)A^{-\beta-\gamma}}\leq M_C$ 
for the constant $M_C>0$ in the statement of the theorem,
and thus the proof is complete.
\end{proof}

  \begin{remark}
    \label{rem:ABBLdiag}
    In the case where $-L$ has a complete set of orthonormal eigenvectors $-L\phi_n=\mu_n\phi_n$  with $0<\mu_1\leq \mu_2\leq \cdots$, the operator $A_0$ in~\eqref{eq:A0D0} has eigenvalues $\lambda_n = \sign(n)i\sqrt{\mu_n}$ and
    a complete set of orthonormal eigenvectors $\set{\psi_n}_{n\in\mathbb{Z}\setminus \set{0}}$ such that
    \eq{
      A_0 \psi_n = \lambda_n \psi_n, \qquad \psi_n = \frac{1}{\sqrt{2}\lambda_n}\pmat{\phi_{|n|}\\\lambda_n\phi_{|n|}}.
    }
    In this situation for every $s\in \R$ the spectral subspace $\ran P_{(s-\delta(s),s+\delta(s))}$  of $A_0$ consists of linear combinations of the eigenvectors $\psi_n$ with every $n\in\mathbb{Z}\setminus \set{0}$ for which $s-\delta(s)<\sign(n)\sqrt{\mu_n}<s+\delta(s)$.
    The functions 
    $\WPgamma:\mathbb{R} \to \WPgammaRan$
    and 
    $\delta:\mathbb{R}\to\deltaRan$
    in Theorem~\ref{thm:kappaestimate} should then be chosen so that 
 \eq{
   \norm{D_0^\ast x_2}\geq \WPgamma(s) \sqrt{\norm{(-L)^{1/2}x_1}_{X_0}^2+\norm{x_2}_{X_0}^2}
 }
 for all $s\in\R$ and $x=(x_1,x_2)^\top\in\ran P_{(s-\delta(s),s+\delta(s))}$.
 In particular, if $\delta:\R\to \deltaRan$ is chosen in such a way that $\delta(-s)=\delta(s)$ and every interval $(i(s-\delta(s)),i(s+\delta(s)))$ contains at most one eigenvalue $\gl_n = \sign(n)i\sqrt{\mu_n}$, 
 then $\WPgamma:\mathbb{R}\to \WPgammaRan$ in Theorem~\ref{thm:kappaestimate} can be chosen to be an even function satisfying 
 \eq{
\norm{D_0^\ast\phi_{|n|}}\geq \sqrt{2}\WPgamma(s) 
   \qquad \mbox{whenever} \quad 
|s-\sqrt{\mu_n}|<\delta(s), \quad s\geq0.
 }
\end{remark}

  \begin{remark}
    \label{rem:kappaestgen}
    The second part of the proof of Theorem~\ref{thm:kappaestimate} can be extended in a straightforward manner to the more general case where $\gb,\gg\geq 0$ are any exponents  satisfying $\gb+\gg\geq \alpha$.
    Indeed, if we denote $n_{\gb\gg}=\ceil{\gb+\gg}$,
    the moment inequality~\cite[Thm.~II.5.34]{EngNag00book} 
    with $\theta=(n_{\gb\gg}-\gb-\gg)/n_{\gb\gg}$ and a constant $M_{\gb,\gg}>0$
    can first be used estimate
    \eq{
      \norm{R(is,A)(-A)^{-\gb-\gg}}\leq M_{\gb,\gg}\norm{R(is,A)}^\theta \norm{R(is,A)A^{-n_{\gb\gg}}}^{1-\theta},
    }
    and  $\norm{R(is,A)A^{-n_{\gb\gg}}}$ can be estimated using the resolvent identity similarly as before.
However, in this case the constant $M_C$ in the bound for $\kappa>0$ has a more complicated formula.
  \end{remark}

\end{subsection}

\begin{subsection}{Robustness results for wave equations}

The structure of the operator $A$ allows us to improve the assumptions of Theorem \ref{T:1.1} to overcome the difficulty of computing the graph norms of the fractional powers of the damped generators $-(A_0-DD^\ast)$ and $-(A_0-DD^\ast)^*$.
Instead, the conditions are given in terms of the graph norms of the fractional powers of the positive operator $-L$.
Throughout this section $C_1^\ast$ denotes the adjoint of $C_1$ as an operator $C_1\colon \Dom C_1\subset X_0\to Y$.

\begin{theorem}\label{T:1.2}
  Assume that the strongly continuous semigroup $T(t)$ generated by 
  \eq{
    A = \pmat{0&I\\L&-D_0D_0^\ast}: \dom L\times \dom (-L)^{1/2}\subset \cH\to \cH
  }
  is polynomially stable with $\alpha\leq 2$, that  $0\leq\beta,\gamma\leq 1$ are such that $\beta+\gamma\geq\alpha$, and that $\kappa>0$ is as in Theorem~\textup{\ref{T:1.1}}. If the perturbation operators $B=(0,B_2)^\top\in\sL(Y,\cH)$ and $C=(C_1,C_2)\in\sL(\cH,Y)$ satisfy
   \begin{equation*}
     \ran B_2\subset\dom(-L)^{\beta/2},\quad \ran C^*_1\subset\dom(-L)^\frac{\gamma-1}{2},\quad \ran C^*_2\subset\dom(-L)^{\gamma/2}
   \end{equation*}
   if $(-L)^{\beta/2} B_2$, $(-L)^\frac{\gamma-1}{2} C^*_1$, and $(-L)^{\gamma/2} C^*_2$ are Hilbert--Schmidt operators, and if
    \begin{equation}\label{E:gen}
        \begin{split}
    &\left\|(-L)^{\beta/2}B_2\right\|<\frac{\kappa}{K_\beta},\\
   &\left\|(-L)^\frac{\gamma-1}{2}C^*_1\right\|^2
    +\left\|(-L)^{\gamma/2}C^*_2\right\|^2<\frac{\kappa^2}{K_\gamma^2},
    \end{split}
    \end{equation}
  then the semigroup generated by $A+BC$ is polynomially stable with the same $\alpha$. Here  $K_\t = e^{\frac{1}{2}{\pi^2\t(1-\t)}} M^\t$ with $\t\in[0,1]$ and $M=1+\|D_0\|^2\|(-L)^{-1/2}\|$.
    \end{theorem}

    For proving this result we use the following theorem from~\cite{Kat61}.
\begin{theorem}[{\cite[Thm. 1]{Kat61}}]\label{C:1}
  If $\mathcal{A}_1$, $\mathcal{A}_2$ are closed maximal accretive operators on a Hilbert space $\sH$ such that $\dom \mathcal{A}_1 \subset \dom \mathcal{A}_2$ and $\|\mathcal{A}_2 u\|\leq M\|\mathcal{A}_1 u\|$ for some constant $M>0$ and for all  $u\in\dom \mathcal{A}_1$, then $\dom \mathcal{A}_1^\t  \subset\dom \mathcal{A}_2^\t  $ and
  \begin{equation*}
    \left\|\mathcal{A}_2^\t  u\right\|\leq K_\t \left\|\mathcal{A}_1^\t  u\right\|, \quad u\in\dom \mathcal{A}_1^\t , \quad 0\leq \t \leq 1,
  \end{equation*}
  where 
  $K_\t = e^{\frac{1}{2}{\pi^2\t (1-\t )}} M^\t$.
\end{theorem}
\begin{proof}[Proof of Theorem~\textup{\ref{T:1.2}}]
  Let $0\leq \beta,\gamma\leq 1$ be such that $\beta+\gamma\geq \alpha$.
  Our aim is to show that if $B_2$, $C_1$, and $C_2$ satisfy the given assumptions, then
  $B=(0,B_2)^\top$ and $C=(C_1,C_2)$ satisfy~\eqref{E:L} with the same $\kappa>0$. The stability of the semigroup generated by  $A+BC$ then follows directly from Theorem~\ref{T:1.1}.
To this end let $\kappa>0$ be as in Theorem~\ref{T:1.1} and suppose that~\eqref{E:gen} hold.
Define $A_d: \Dom A_d\subset \cH \to \cH$ and $A_0: \Dom A_0\subset \cH \to \cH$ with domains $\Dom A_d = \Dom A_0 = \Dom A = \Dom L\times \Dom (-L)^{1/2}$ by
\begin{align*}
  A_d:=\begin{pmatrix}
    (-L)^{1/2}&0\\
    0&(-L)^{1/2}
  \end{pmatrix} \qquad \mbox{and} \qquad
  A_0:=\begin{pmatrix}
    0&I\\
    L& 0 \end{pmatrix},
\end{align*}
and let $D=(0,D_0)^\top\in \Lin(U,\cH)$.
Clearly
\begin{align*}
  \Dom A_d^\theta = \Dom(-L)^{(\theta+1)/2}\times \Dom(-L)^{\theta/2}
\end{align*}
for all $0\leq \theta\leq 1$.
Since $C\in \Lin(\cH,Y)$ and $C^\ast = (-L^{-1} C_1^\ast,C_2^\ast)^\top$,
the assumptions on $B_2$, $C_1$, and $C_2$ imply that $\ran B\subset \Dom A_d^\beta$ and $\ran C^\ast \subset \Dom A_d^\gamma$.
For every $u=(u_1,u_2)^\top\in \Dom A_d $ we have
$\norm{A_0u}_\cH^2 = \norm{(-L)^{1/2}u_2}^2+\norm{Lu_1}^2 = \norm{A_d u}_\cH^2$, and
  \begin{equation*}
  \begin{split}
    \|Au\|_\cH&=\left\|(A_0-DD^*)u\right\|_\cH\leq\left\|I-DD^*A_0^{-1}\right\|\|A_0u\|_\cH\\
    &\leq \left(1+\left\|DD^*A_0^{-1}\right\|\right)\|A_du\|_\cH,
  \end{split}
 \end{equation*}
 where $\| DD^\ast A_0^{-1}\| \leq \norm{D_0}^2 \norm{(-L)^{-1/2}} $.
 An analogous argument shows that we have $\norm{A^\ast u}_\cH = \norm{(A_0+DD^\ast) u}\leq (1+\norm{D_0}^2 \norm{(-L)^{-1/2}})\norm{A_d u}_\cH$ for all $u\in \Dom A_d$.
Since $-A$, $-A^\ast$ and $A_d$ are closed and maximally accretive operators and $\Dom A = \Dom A^\ast = \Dom A_d$, Theorem~\ref{C:1} implies that
$\ran B\subset \Dom(-A)^\beta$ and $\ran C^\ast \subset \Dom(-A^\ast)^\gamma$, and
for all $y\in Y$
 with $\norm{y}=1$
 we have
 \begin{align*}
   \norm{(-A)^\beta B y}_\cH
   &\leq K_\beta \norm{A_d^\beta By}_\cH
   \leq K_\beta \norm{(-L)^{\beta/2} B_2 }\norm{y} <\kappa\\
   \norm{(-A^\ast)^\gamma C^\ast y}_\cH
   &\leq K_\gamma \norm{A_d^\gamma C^\ast y}_\cH\\
   &\leq K_\gamma \left(\norm{(-L)^{(\gamma-1)/2} C_1^\ast }^2 + \norm{(-L)^{\gamma/2} C_2^\ast }^2\right)^{1/2} \norm{y}
   <\kappa.
 \end{align*}
 By Theorem~\ref{T:1.1} the semigroup generated by $A+BC$ is polynomially stable with $\alpha$.
\end{proof}

If the operator $L$ is diagonalizable~\cite[Sec.~2.6]{TucWei09book}, then for $\t\in \mathbb{R}$ the spaces $\Dom(-L)^\t$ and the graph norms of $(-L)^\t$ have the forms
\begin{subequations}
  \label{E:Hspace}
  \begin{align}
    H_{\t}(L)&:= \Dom(-L)^\t= \left\{u\in X_0:\sum\limits_{k=1}^\infty\mu_{k}^{2\t}|\langle u,\phi_{k}\rangle_{X_0}|^2<\infty\right\}\\
    \|u\|_{H_\t}&:= \norm{(-L)^\t u}_{X_0}=\sum\limits_{k=1}^\infty\mu_{k}^{2\t}|\langle u,\phi_{k}\rangle_{X_0}|^2, \qquad u\in H_\t(L),
  \end{align}
\end{subequations}
where $\mu_{k}$ are the eigenvalues of $-L$ and
$\phi_{k}$ are the corresponding orthonormal eigenvectors.
With these definitions the space $H_{-\t}(L)$ is the dual of $H_\t(L)$ with respect to the pivot space $X_0$~\cite[Sec. 2.9]{TucWei09book}.

\begin{corollary}\label{C:mainth}
  Assume that $L$ is diagonalisable, that the strongly continuous semigroup $T(t)$ generated by $A$ is polynomially stable with $\alpha\leq 2$, $0\leq\beta,\gamma\leq 1$ satisfy $\beta+\gamma\geq\alpha$, and $\kappa>0$ is as in Theorem~\textup{\ref{T:1.1}}. If the perturbation operators $B=(0,B_2)^\top\in\sL(Y,\cH)$ and $C=(C_1,C_2)\in\sL(\cH,Y)$ satisfy
\begin{align*}
  B_2\in \Lin(Y,H_{\beta/2}(L)), \qquad C_1^\ast \in \Lin(Y,H_{(\gamma-1)/2}(L)), \quad \mbox{and} \quad C_2^\ast \in \Lin(Y,H_{\gamma/2}(L)),
\end{align*}
 if $(-L)^{\beta/2} B_2$, $(-L)^\frac{\gamma-1}{2} C^*_1$, and $(-L)^{\gamma/2} C^*_2$ 
are Hilbert--Schmidt operators and if
    \begin{equation*}
        \begin{split}
	  &\left\|B_2\right\|_{\Lin(Y,H_{\beta/2})}<\frac{\kappa}{K_\beta},\\
	  &\left\|C^*_1\right\|^2_{\Lin(Y,H_{(\gamma-1)/2})}
	  +\left\|C^*_2\right\|^2_{\Lin(Y,H_{\gamma/2})}<\frac{\kappa^2}{K_\gamma^2},
    \end{split}
    \end{equation*}
  then the semigroup generated by $A+BC$ is polynomially stable with the same $\alpha$. Here $K_\t = e^{\frac{1}{2}{\pi^2\t(1-\t)}} M^\t$ with $\t\in[0,1]$ and $M=1+\|D_0\|^2\|(-L)^{-1/2}\|$.
\end{corollary}

\end{subsection}

 \section{Perturbations of damped two-dimensional wave equations}
 \label{sec:2Dwave}

In this section we consider damped wave equations on rectangular domains with different damping functions. We use Theorem~\ref{T:1.2} to derive concrete conditions for preservation the polynomial stability of perturbed wave equations with finite rank and Hilbert-Schmidt perturbations.
We consider the damped wave equation~\eqref{E:DWE} on $\Omega = (0,a)\times(0,b)$, $a,b>0$, with a damping coefficient $d(\cdot,\cdot)\in L^\infty(\Omega)$.
The equation is of the form~\eqref{GWE} on $X_0=L^2(\Omega)$ with the choice 
  $L=\Delta$ and domain $\dom\Delta=H^2(\Omega)\cap H_0^1(\Omega)$, and with $D_0\in \Lin(L^2(\Omega))$ defined 
  as the multiplication operator such that $D_0u=\sqrt {d(\cdot,\cdot)}u(\cdot,\cdot)$ for all $u\in L^2(\Omega)$.

We suppose that the set $\omega=\{d(x,y)>0\}$ contains an open, nonempty subset and does not satisfy Geometric Control Condition (GCC) (see a definition of GCC for example in \cite[Sec.~1]{AnaLea14}). It was shown in \cite{Jaf90} that for such damping the Schr\"odinger group is observable, i.e., the pair $(D_0^*, i(-\Delta))$ is exactly observable~\cite[Def.~6.1.1]{TucWei09book} (see also~\cite{BurZwo19}). 
In this case
the damped wave equation \eqref{E:DWE} is polynomially stable with $\alpha= 2$ by~\cite[Thm.~2.3]{AnaLea14}.

  Our assumptions together with the results in~\cite{Jaf90} and~\cite[Prop.~3.9]{ChiPau19arxiv} also imply that the condition~\eqref{eq:WPestimate} is satisfied for some functions   
 $\WPgamma :\R\to(0,\WPgamma_0] $ and $\delta :\mathbb{R}\to (0,\delta_0]$
 satisfying
$\WPgamma(s)^{-2}\delta(s)^{-2}\leq M_0(1+s^2) $ for all $s\in\R$.
Because of this, Theorem~\ref{thm:kappaestimate} could in principle be used to derive numerical values for $\kappa>0$ for particular damping functions $d(\cdot,\cdot)$.  In practice, however, finding suitable concrete functions $\WPgamma$ and $\delta$ can be challenging,
and 
in the case of the two-dimensional wave equation this is an important topic for further research.

 \begin{remark}\label{strip}
   In some cases of damping functions the estimate for the exponent of polynomial stability can be improved. For example, in \cite{Sta17}  the exponent of polynomial stability for the damping function
  \begin{equation}
    \label{eq:stripdamping}
  d(x,y)=
  \left\{
  \begin{array}{ll}
    1    & \mbox{if $x< \e$};\\
    0         & \mbox{if $x> \e$},
  \end{array}\right.\quad \e\in(0,1)
\end{equation}
was shown to be $\a=3/2$. Moreover, additional differentiability assumptions on $d(\cdot,\cdot)$ improve the rate of polynomial decay, as shown in~\cite{BurHit07,AnaLea14,DatKle20}.
\end{remark}
\begin{subsection}{Rank one perturbations}
 We begin by considering perturbed wave equations of the form

    \begin{equation}\label{E:pW}
    \begin{split}
      &w_{tt}(t,x,y)-\Delta w(t,x,y)+d(x,y)w_t(t,x,y)\\
      &=b_2(x,y)\int\limits_\Omega (w(t,\xi,\eta) c_1(\xi,\eta)+ w_t(t,\xi,\eta)c_2(\xi,\eta))d\xi d\eta
      \end{split}
    \end{equation}
    with $b_{2},c_{2}\in L^2(\Omega)$ and $c_{1}\in H_{-1/2}(\Delta)$.
The following theorem presents sufficient conditions for the polynomial stability of~\eqref{E:pW}.

\begin{theorem}\label{T:rank1}
Assume that damped wave equation \eqref{E:DWE} is polynomially stable with $\a\leq 2$, $0\leq\beta,\gamma\leq 1$ are such that $\beta+\gamma\geq\a$,
     and $\kappa>0$ is as in Theorem~\textup{\ref{T:1.2}}.
If $b_2\in H_{\beta/2}(\Delta)$, $c_1\in H_{{(\gamma-1)/2}}(\Delta)$, $c_2\in H_{\gamma/2}(\Delta)$
  satisfy
   \begin{equation}\label{E:rank2}
    \|b_2\|_{H_{\beta/2}}<\frac{\kappa}{K_\beta},\quad \|c_1\|^2_{H_{(\gamma-1)/2}}+\|c_2\|^2_{H_{\gamma/2}}<\frac{\kappa^2}{K_\gamma^2},
  \end{equation}
then perturbed wave equation \eqref{E:pW} is polynomially stable with the same $\a$. 
Here $K_\theta = e^{\frac{1}{2}{\pi^2\theta(1-\theta)}} M^\theta$ with $M = 1+\frac{ab \|d\|_{L^\infty}}{\pi\sqrt{a^2+b^2}}$ and $\t\in[0,1]$.  
For such perturbations there exists $\Mpert>0$ such that the solutions of~\eqref{E:pW} corresponding to initial conditions $w_0\in H^2(\Omega)\cap H_0^1(\Omega)$ and $w_1\in H_0^1(\Omega)$ satisfy
\begin{align*}
  \norm{w(t,\cdot,\cdot)}_{H^1}^2 + \norm{w_t(t,\cdot,\cdot)}_{L^2}^2\leq \frac{\Mpert}{t^{2/\a}}\left( \norm{w_0}_{H^2}^2 + \norm{w_1}_{H^1}^2 \right), \qquad t>0.
\end{align*}
\end{theorem}
\begin{proof}
    In this case the perturbed wave equation has the form~\eqref{PGWE2} with $Y=\mathbb{C}$ and
\begin{align*}
  B_2 = b_2\in L^2(\Omega), \qquad C_1 = \langle\cdot,c_1\rangle_{H_{1/2},H_{-1/2}}, \qquad \mbox{and} \qquad C_2 = \langle\cdot,c_2\rangle_{L^2},
\end{align*}
where $\langle\cdot,\cdot\rangle_{H_{1/2},H_{-1/2}}$ denotes the dual pairing between
$H_{1/2}(\Delta)$
and $H_{-1/2}(\Delta)$.
Since $B_2=b_2$, $C_1^\ast = c_1$ and $C_2^\ast = c_2$, the claim follows from Corollary~\ref{C:mainth} and a suitable upper bound for $\|D_0\|^2\left\|(-\Delta)^{-1/2}\right\|$.
  Since $-\Delta$ is a positive self-adjoint operator with compact resolvent and its smallest eigenvalue is $\pi^2 (a^2+b^2)/(a^2b^2)$, we have $\left\|(-\Delta)^{-1/2}\right\|\leq\frac{ab}{\pi\sqrt{a^2+b^2}}$.
Since $\|D_0\|= \|\sqrt {d(\cdot,\cdot)}\|_{L^\infty}=\sqrt{\| d(\cdot,\cdot)\|_{L^\infty}}$, the claim holds for $M = 1+\frac{ ab\|d\|_{L^\infty}}{\pi\sqrt{a^2+b^2}}$.
\end{proof}

\begin{remark}
  Conditions \eqref{E:rank2} have the simplest form if we choose $\beta=\gamma=1$
  \begin{equation*}
    \|b_2\|_{H_{1/2}}<\frac{\kappa}{M},\quad \|c_1\|^2_{L^2}+\|c_2\|^2_{H_{1/2}}<\frac{\kappa^2}{M^2},
  \end{equation*}
  where $M = 1+\frac{ ab\|d\|_{L^\infty}}{\pi\sqrt{a^2+b^2}}$.
\end{remark}

\begin{remark}
  If the damping function is as in~\eqref{eq:stripdamping}, then the exponent of polynomial stability is $\a=3/2$. In this case  $\beta$, $\gamma$ can be chosen as $\beta=\frac{1}{2}$ and $\gamma=1$ and conditions \eqref{E:rank2} take the form
  \begin{equation*}
    \|b_2\|_{H_{1/4}}<\frac{\kappa}{K_{1/2}},\quad \|c_1\|^2_{L^2}+\|c_2\|^2_{H_{1/2}}<\frac{\kappa^2}{M^2},
  \end{equation*}
  where $K_{1/2}=e^{\frac{\pi^2}{8}}\sqrt M$ with $M = 1+\frac{ ab}{\pi\sqrt{a^2+b^2}}$.
\end{remark}
\end{subsection}
\begin{subsection}{Finite rank perturbations}
  We now consider the wave equation~\eqref{E:DWE} with a finite number of perturbation terms
\begin{equation}\label{E:pWf}
    \begin{split}
      &w_{tt}(t,x,y)-\Delta w(t,x,y)+d(x,y)w_t(t,x,y)\\
      &=\sum\limits_{k=1}^m b_{k,2}(x,y)\int\limits_\Omega \left(w(t,\xi,\eta) c_{k,1}(\xi,\eta)+ w_t(t,\xi,\eta)c_{k,2}(\xi,\eta)\right)d\xi d\eta
      \end{split}
    \end{equation}
    with $b_{k,2},c_{k,2}\in L^2(\Omega)$ and $c_{k,1}\in H_{-1/2}(\Delta)$. 

\begin{theorem}
Assume that damped wave equation \eqref{E:DWE} is polynomially stable with $\a\leq 2$, $0\leq\beta,\gamma\leq 1$ are such that $\beta+\gamma\geq\a$,
and $\kappa>0$ is as in Theorem~\textup{\ref{T:1.2}}.
If for all $k\in \{1,\ldots,m\}$ we have
  $b_{k,2}\in H_{\beta/2}(\Delta)$, $c_{k,1}\in H_{{(\gamma-1)/2}}(\Delta)$, and $c_{k,2}\in H_{\gamma/2}(\Delta)$
  and
  \begin{equation}\label{E:rank_m}
  \begin{split}
   &\left\|c_{k,1}\right\|^2_{H_{{(\gamma-1)/2}}}
   +\left\|c_{k,2}\right\|^2_{H_{\gamma/2}}
   <\frac{\kappa^2}{m^2K^2_\gamma},\\
  &\left\|b_{k,2}\right\|_{H_{\beta/2}}<\frac{\kappa}{mK_\beta},
  \end{split}
  \end{equation}
 then the perturbed wave equation \eqref{E:pWf} is polynomially stable with the same $\a$.
 Here $K_\theta = e^{\frac{1}{2}{\pi^2\theta(1-\theta)}} M^\theta$ with
 $M = 1+\frac{ ab\|d\|_{L^\infty}}{\pi\sqrt{a^2+b^2}}$ and $\t\in[0,1]$.
For such perturbations there exists $\Mpert>0$ such that the solutions of~\eqref{E:pWf} corresponding to initial conditions $w_0\in H^2(\Omega)\cap H_0^1(\Omega)$ and $w_1\in H_0^1(\Omega)$ satisfy
\begin{align*}
  \norm{w(t,\cdot,\cdot)}_{H^1}^2 + \norm{w_t(t,\cdot,\cdot)}_{L^2}^2\leq \frac{\Mpert}{t^{2/\a}}\left( \norm{w_0}_{H^2}^2 + \norm{w_1}_{H^1}^2 \right), \qquad t>0.
\end{align*}
 
\end{theorem}
\begin{proof}
    The perturbation can be written in the form~\eqref{PGWE2} with the choice $Y=\mathbb{C}^m$ and defining $B_2\in \Lin(Y,X_0)$, $C_1\in \Lin(\Dom(-L)^{1/2},Y)$, and $C_2\in \Lin(X_0,Y)$ so that
    \begin{align*}
      B_2 y &= \sum_{k=1}^m b_{k,2}y_k, \\
      C_1x &= (\langle x,c_{k,1}\rangle_{H_{1/2},H_{-1/2}})_{k=1}^m\in Y, \\
      C_2z &= (\langle z,c_{k,2}\rangle_{L^2})_{k=1}^m\in Y
    \end{align*}
    for all $ y=(y_k)_{k=1}^m \in Y$, $x\in H_{1/2}(\Delta)$ and $z\in L^2(\Omega)$.
  For any $0\leq\beta,\gamma\leq1$ we have
  \begin{align*}
    \|B_2\|_{\Lin(\dC^m,H_{\beta/2})}
    &\leq\sum_{k=1}^m\|b_{k,2}\|_{H_{\beta/2}}< m\frac{\kappa}{mK_\beta}=\frac{\kappa}{K_\beta}\\
    \left\| C^*_1\right\|^2_{\Lin(\dC^m,H_{(\gamma-1)/2})}
    +\left\|C^*_2\right\|^2_{\Lin(\dC^m,H_{\gamma/2})}
    &\leq
    m\sum_{k=1}^m\left(\left\|c_{k,1}\right\|^2_{H_{{(\gamma-1)/2}}}
      +\left\|c_{k,2}\right\|^2_{H_{\gamma/2}}   \right)\\
    & <m^2\frac{\kappa^2}{m^2K_\gamma^2}=\frac{\kappa^2}{K_\gamma^2},
  \end{align*}
   and thus the claims follow from Corollary~\ref{C:mainth} as in the proof of Theorem~\ref{T:rank1}.
   \end{proof}
\end{subsection}

\begin{subsection}{Hilbert--Schmidt perturbations}
   Now we consider a more general case of perturbations of the wave equation

   \begin{equation}\label{E:pWH}
    \begin{split}
      &w_{tt}(t,x,y)-\Delta w(t,x,y)+d(x,y)w_t(t,x,y)\\
      &=\sum\limits_{k=1}^\infty b_{k,2}(x,y)\int\limits_\Omega \left(w(t,\xi,\eta) c_{k,1}(\xi,\eta)+ w_t(t,\xi,\eta)c_{k,2}(\xi,\eta)\right)d\xi d\eta
      \end{split}
    \end{equation}
    where the functions $b_{k,2},c_{k,2}\in L^2(\Omega)$ and $c_{k,1}\in H_{-1/2}(\Delta)$ of the perturbation are assumed to satisfy
\begin{align*}
  \sum\limits_{k=1}^\infty\|b_{k,2}\|^2_{L_2(\Omega)}<\infty, \quad
  \sum\limits_{k=1}^\infty\|c_{k,1}\|^2_{H_{-1/2}}<\infty, \quad
  \mbox{and} \quad \sum\limits_{k=1}^\infty\|c_{k,2}\|^2_{L_2(\Omega)}<\infty.
\end{align*}
The stability of this perturbed wave equation can be studied using Corollary~\ref{C:mainth} for Hilbert--Schmidt perturbations.

\begin{theorem}
Assume that damped wave equation \eqref{E:DWE} is polynomially stable with $\a\leq 2$, $0\leq\beta,\gamma\leq 1$ are such that $\beta+\gamma\geq\a$,
and $\kappa>0$ is as in Theorem~\textup{\ref{T:1.2}}.
If for all $k\in \mathbb{N}$ we have
$b_{k,2}\in H_{\beta/2}(\Delta)$, $c_{k,1}\in H_{{(\gamma-1)/2}}(\Delta)$, and $c_{k,2}\in H_{\gamma/2}(\Delta)$
       and
\begin{align*}
  &\sum_{k=1}^\infty\left\|b_{k,2}\right\|^2_{H_{\beta/2}}
    <\frac{\kappa^2}{K_\beta^2 },\\
    &\sum_{k=1}^\infty \left\|c_{k,1}\right\|^2_{H_{{(\gamma-1)/2}}}
     +\left\|c_{k,2}\right\|^2_{H_{\gamma/2}}
     <\frac{ \kappa^2}{K^2_\gamma },
\end{align*}
 then the perturbed wave equation \eqref{E:pWH} is polynomially stable with the same $\a$.
 Here $K_\theta = e^{\frac{1}{2}{\pi^2\theta(1-\theta)}} M^\theta$ with
 $M = 1+\frac{ ab\|d\|_{L^\infty}}{\pi\sqrt{a^2+b^2}}$ and $\t\in[0,1]$.
For such perturbations there exists $\Mpert>0$ such that the solutions of~\eqref{E:pWH} corresponding to initial conditions $w_0\in H^2(\Omega)\cap H_0^1(\Omega)$ and $w_1\in H_0^1(\Omega)$ satisfy
\begin{align*}
  \norm{w(t,\cdot,\cdot)}_{H^1}^2 + \norm{w_t(t,\cdot,\cdot)}_{L^2}^2\leq \frac{\Mpert}{t^{2/\a}}\left( \norm{w_0}_{H^2}^2 + \norm{w_1}_{H^1}^2 \right), \qquad t>0.
\end{align*}
\end{theorem}

\begin{proof}
The perturbations can be written in the form~\eqref{PGWE2} with $Y=\ell^2(\dC)$ 
if we define $B_2$, $C_1$, and $C_2$ so that
    \begin{align*}
      B_2 y &= \sum_{k=1}^\infty b_{k,2}y_k, \\
      C_1x &= (\langle x,c_{k,1}\rangle_{H_{1/2},H_{-1/2}})_{k=1}^\infty, \\
      C_2z &= (\langle z,c_{k,2}\rangle_{L^2})_{k=1}^\infty,
    \end{align*}
    The assumptions
    $(\norm{b_{k,2}}_{L^2})_k\in\ell^2$, $(\norm{c_{k,1}}_{H_{-1/2}})_k\in\ell^2$, and $(\norm{c_{k,2}}_{L^2})_k\in\ell^2$ 
     imply that $B_2\in \Lin(Y,X_0)$, $C_1\in \Lin(H_{1/2},Y)$ and $C_2\in \Lin(X_0,Y)$.
If we let $0\leq\beta,\gamma\leq1$ be such that $\beta+\gamma\geq\a$, then 
\begin{align*}
   &\sum\limits_{k=1}^\infty \| (-\Delta)^{\beta/2}b_{k,2} \|_{X_0}^2
  =\sum\limits_{k=1}^\infty \| b_{k,2} \|_{H_{\beta/2}}^2
       <
\frac{\kappa}{ K_\beta}\\
&\sum_{k=1}^\infty\left(\left\|(-\Delta)^{(\gamma-1)/2}c_{k,1}\right\|^2_{X_0}
  +\left\|(-\Delta)^{\gamma/2}c_{k,2}\right\|^2_{X_0}   \right)\\
&\hspace{3cm}=
\sum_{k=1}^\infty\left(\left\|c_{k,1}\right\|^2_{H_{{(\gamma-1)/2}}}
   +\left\|c_{k,2}\right\|^2_{H_{\gamma/2}}   \right)
 <
 \frac{\kappa^2}{ K_\gamma^2}.  
\end{align*}
imply that 
$(-\Delta)^{\beta/2} B_2$, $(-\Delta)^\frac{\gamma-1}{2} C^*_1$, and $(-\Delta)^{\gamma/2} C^*_2$ are Hilbert--Schmidt operators and
that 
\begin{equation*}
  \left\|B_2\right\|_{\Lin(Y,H_{\beta/2})}<\frac{\kappa}{K_\beta}, \qquad \mbox{and} \qquad
  \left\|C^*_1\right\|^2_{\Lin(Y,H_{(\gamma-1)/2})}
  +\left\|C^*_2\right\|^2_{\Lin(Y,H_{\gamma/2})}<\frac{\kappa^2}{K_\gamma^2}.
\end{equation*}
Thus the claim follows from Corollary~\ref{C:mainth} as in the proof of Theorem~\ref{T:rank1}.
     \end{proof}
\end{subsection}

\subsection{Wave equation with "almost dissipative" damping}\label{S:ADD}

Finally, we consider the two-dimensional damped wave equation with a perturbed damping term, namely
 \begin{equation}\label{E:pWad}
    \begin{split}
      &w_{tt}(t,x,y)-\Delta w(t,x,y)+d(x,y)w_t(t,x,y)\\
      &- b_2(x,y)\int\limits_\Omega \sqrt {d(\xi,\eta)}w_t(t,\xi,\eta) c(\xi,\eta)d\xi d\eta=0
      \end{split}
    \end{equation} 
 with $b_2,c\in L^2(\Omega)$. 
    We also make an additional assumption that  $d\in C^2(\Omega)$.
The structure of the perturbed semigroup generator is now
 $\wt A:=A_0-(D+B)D^*=A-BD^*$, where $D=(0,\sqrt {d(\cdot)})^\top$  and $B=(0,b_2\langle\cdot,c\rangle_{L^2})^\top$.
 Because of this structure, the damping in the wave equation~\eqref{E:pWad} can be thought to be ``almost dissipative''.

 Since we assumed that the damping coefficient is smooth, i.e. $d(\cdot)\in C^2(\Omega)$, it is possible to characterise the higher order domain $\Dom A^2$ and
the stability of~\eqref{E:pWad} can be studied using
Theorem~\ref{T:1.1} with the parameters $\beta=2$ and $\gamma=0$, as shown in the following theorem.
\begin{theorem}
Assume that damped wave equation \eqref{E:pWad} is polynomially stable with $\a\leq 2$ in the case where $b_2=0$. 
  There exists $\kappa>0$ such that if
  $b_2\in\dom\Delta$ and $c\in L^2(\Omega)$ satisfy
  \begin{equation}\label{E:aldd}
    \|\sqrt dc\|_{L^2}<\kappa,\quad \|db_2\|_{H_{1/2}}^2+\| b_2\|_{H^1}^2<\frac{\kappa^2}{M^2},
  \end{equation}
  where $M=1+\frac{ab\|d\|^2_{L^\infty}}{\pi\sqrt{a^2+b^2}}$, then \eqref{E:pWad} is polynomially stable with the same $\a$.
For such perturbations there exists $\Mpert>0$ such that the solutions of~\eqref{E:pWad} corresponding to initial conditions $w_0\in H^2(\Omega)\cap H_0^1(\Omega)$ and $w_1\in H_0^1(\Omega)$ satisfy
\begin{align*}
  \norm{w(t,\cdot,\cdot)}_{H^1}^2 + \norm{w_t(t,\cdot,\cdot)}_{L^2}^2\leq \frac{\Mpert}{t^{2/\a}}\left( \norm{w_0}_{H^2}^2 + \norm{w_1}_{H^1}^2 \right), \qquad t>0.
\end{align*}
\end{theorem}
\begin{proof}
  Let $\kappa>0$ be as in Theorem~\ref{T:1.1} and suppose the assumptions on $b_2$ and $c$ are satisfied. We define $\wt B:=(0,b_2)^\top$ and $\wt C:=(0,\langle\cdot,\sqrt dc\rangle)$. It is clear that $\wt B\wt C=BD^*$.
  Our aim is to verify that the conditions of  Theorem~\ref{T:1.1} are satisfied for the perturbed operator for $A-\wt B\wt C$ with parameters $\beta=2$ and $\gamma=0$. 
We have
  \begin{equation*}
    \begin{split}
        \dom A^2&=\{u\in\dom A_0:Au\in\dom A_0\}\\
        &=\left\{u=\begin{pmatrix}u_1\\u_2\end{pmatrix}\in
        \begin{pmatrix}\dom \Delta\\\dom(-\Delta)^{1/2}\end{pmatrix}:
        \begin{pmatrix}u_2\\\Delta u_1-du_2\end{pmatrix}\in
        \begin{pmatrix}\dom \Delta\\\dom(-\Delta)^{1/2}\end{pmatrix}\right\},
    \end{split}
  \end{equation*}
and thus $\ran B= \{0\}\times \spann \{b_j\}
\subset\dom A^2$ provided that $b_2\in\dom\Delta$ and $db_2\in \dom(-\Delta)^{1/2}$. Since $d\in C^2(\Omega)$, the assumption $b_2\in\dom\Delta$ also implies $db_2\in \dom(-\Delta)^{1/2}$.

  The norm of $A^2B$ can be estimated by
  \begin{equation}\label{E:a2b}
  \begin{split}
    \|A^2B\|^2&\leq\|I-DD^*A_0^{-1}\|^2\|A_0AB\|^2\\
    &\leq(I+\|d\|_{L^\infty}\|(-\Delta)^{-1/2}\|)^2\|A_0AB\|^2\leq M^2 \|A_0AB\|^2,
    \end{split}
  \end{equation}
  where the last estimate is completed as in the proof of Theorem \ref{T:rank1}.
Moreover,
  \begin{equation}\label{E:a0ab}
  \begin{split}
    \|A_0AB\|^2&\leq\left\|
    \begin{pmatrix}
      \Delta&-d\\
      0 &\Delta
    \end{pmatrix}
    \begin{pmatrix}
      0\\
      b_2
    \end{pmatrix}
    \right\|^2=\left\|\begin{pmatrix}
      db_2\\
      \Delta b_2
    \end{pmatrix}
    \right\|^2\\
    &=(\|(-\Delta)^{1/2}(db_2)\|^2+\|\Delta b_2\|^2)<\frac{\kappa^2}{M^2}.
    \end{split}
  \end{equation}
 Thus $\|(-A)^2B\|<\kappa$.
 We also have that $\|(-A^*)^0C^*\|=\|\sqrt dc\|_{L^2}<\kappa$.
 The polynomial stability of the semigroup generated by $A-\wt B\wt C=A- BD^*$ follows from Theorem \ref{T:1.1} and then wave equation~\eqref{E:pWad} is polynomial stable with $\a$.
 \end{proof}

\section{Perturbations of Webster's equations}
In this section we show the polynomial stability of weakly damped Webster's equation and use Theorems~\ref{T:1.2} and~\ref{thm:kappaestimate} to derive sufficient conditions for the preservation of the stability under addition of perturbing terms.
We begin by considering an undamped Webster's equation on $\Omega=(0,1)$ which has the form
\begin{equation*}
  \left\{
  \begin{array}{ll}
    w_{tt}(t,x)
    =\frac{1}{r(x)}\left(r(x)w_x(t,x)\right)_x    \\
    w(t,0)=w(t,1)=0  \\
    w(0,x)=w_0(x),\quad w_t(0,x)=w_1(x).
  \end{array}\right.
\end{equation*}
We consider $r(x)=e^{ax}$, where $a\geq 0$.
Then Webster's equation takes the form
\begin{equation*}
    w_{tt}(t,x) = w_{xx}(t,x)+aw_x(t,x).
\end{equation*}
We denote by $L_a^2(0,1)$ the Hilbert space $L^2(0,1)$ with the inner product
$$
\langle f,g\rangle_{L_a^2}=\int\limits_0^1 f(\xi)\ov {g(\xi)}e^{a\xi}d\xi.
$$
Let us define the operator $L=\frac{d^2 }{d x^2}+a\frac{d }{d x}$ from $L^2_a(0,1)$ to $L^2_a(0,1)$ with $\dom L=\{h\in L^2_a(0,1):h,h' \text{ are absolutely continuous }, h''\in L^2_a(0,1) \text{ and } h(0)=h(1)=0\}$.
In the next lemma we state some properties of $L$.  
\begin{lemma}
 The operator $L=\frac{d^2 }{d x^2}+a\frac{d }{d x}$ from $L^2_a(0,1)$ to $L^2_a(0,1)$ is
a negative self-adjoint operator with a bounded inverse.
 The eigenvalues and eigenvectors of $L$ are
 $$
 \m_n=-\frac{a^2}{4}-\pi^2n^2,\quad \varphi_n(x)=e^{-\frac{ax}{2}}\sin(\pi nx),
 $$
 respectively, for $n\in\mathbb N$.
\end{lemma}
\begin{proof}
  We define a unitary mapping $V:L^2(0,1)\to L^2_a(0,1)$ by the formula
  \begin{equation*}
    (Vf)(x)=e^{-\frac{ax}{2}}f(x), \qquad x\in(0,1).
  \end{equation*}
  Now we can consider an auxiliary operator $\wt L:L^2(0,1)\to L^2(0,1)$ defined by $\wt L= V^*LV$ with $\dom \wt L=H^2(0,1)\cap H^1_0(0,1).$
   Direct calculations yield that $\wt Lf=\frac{d^2 f }{d x^2}-\frac{a^2}{4}f$. It is well known that $\wt L$ is 
a negative self-adjoint operator with a bounded inverse.
 Hence $L$ is also a negative self-adjoint operator with a bounded inverse.
 The eigenvalues and the eigenvectors of the operator $\wt L$ are
  $$
  \mu_n=-\frac{a^2}{4}-\pi^2n^2,\quad \wt \varphi_n(x)=\sin(\pi nx) ,\qquad n\in\N.
 $$
 Since 
$\wt L= V^*LV$, 
the operators $\wt L$ and $L$ have the same eigenvalues
 and the eigenvectors of $L$ are given by the formula $\varphi_n(x)=(V\wt \varphi_n)(x)=e^{-\frac{ax}{2}}\sin(\pi nx)$.
  \end{proof}

Now we consider weakly damped Webster's equation
\begin{equation}\label{web_weak}
  \left\{
  \begin{array}{ll}
    w_{tt}(t,x)
    - w_{xx}(t,x)-aw_x(t,x)
    +d(x)\int_0^1w_t(t,\xi)d(\xi)e^{a\xi} d\xi =0    \\
    w(t,0)=w(t,1)=0  \\
    w(0,x)=w_0(x),\quad w_t(0,x)=w_1(x),
  \end{array}\right.
\end{equation}
where the damping coefficient is $d\in L^2_a(0,1)$.
This equation is of the form~\eqref{GWE} on $X_0=L_a^2(0,1)$ with  
  $L$ defined above
  and with a rank one operator $D_0 = d(\cdot)\in \Lin(\dC,L_a^2(0,1))$ and
$D_0^*=\langle \cdot,d\rangle_{L^2_a}$.

The polynomial stability of the weakly damped Webster's equation can be analyzed using~\cite[Thm.~6.3]{Pau17b}.
The following result in particular shows that~\eqref{web_weak} is polynomially stable for the particular choice of damping $d(x)=1-x$.

  \begin{proposition}\label{CorWeb}
    The weakly damped Webster's equation \eqref{web_weak} with the damping function $d(x)=1-x$ is polynomially stable with $\a=2$.
  \end{proposition}

  \begin{proof}
    We can write
    \begin{align*}
      A = 
      \begin{pmatrix}
	0&I\\L&-D_0D_0^\ast
      \end{pmatrix}
      =
      \begin{pmatrix}
	0&I\\L&0
      \end{pmatrix}
      -
      \begin{pmatrix}
	0\\D_0
      \end{pmatrix}
      \begin{pmatrix}
	0&D_0^\ast
      \end{pmatrix}
      =: A_0-DD^\ast
    \end{align*}
    with $\dom A_0 = \dom A$ and $D\in \Lin(\dC,X)$.
    We will use~\cite[Thm. 6.3]{Pau17b} to show that $\norm{R(is,A)}\leq M(1+s^2)$ for some $M>0$.
    To this end, we need to estimate the quantities $|D^\ast \psi_n|$ from below, where $\psi_n$ are the normalized eigenvectors of $A_0$. 
Since $L$ has eigenvalues $\mu_n=-\frac{a^2}{4}-\pi^2n^2$ with the corresponding eigenvectors $\varphi_n(x)=e^{-\frac{ax}{2}}\sin(\pi nx)$ for $n\in\mathbb N$, the
eigenvectors $\psi_n$ and the
corresponding eigenvalues $\l_n$ of $A_0$ are given by 
$$
\l_{n}=\sign(n) i\sqrt{\frac{a^2}{4}+\pi^2n^2}, \qquad \mbox{and} \qquad
\psi_n(x)=\frac{1}{\l_n}\begin{pmatrix}
  \varphi_{|n|}(x)\\
  \l_n \varphi_{|n|}(x)
\end{pmatrix}
  , \quad n\in\mathbb Z\backslash\{0\}.
$$
For any $n\in \mathbb{Z}\setminus \{0\}$ we thus have
    \begin{equation*}
      \begin{split}
	|D^*\psi_n|&=\left|\langle\varphi_{|n|}(x),
	1-x\rangle_{L_a^2}\right|
	=\left|\int\limits_0^1 e^{ax}e^{-\frac{a}{2}x}\sin(\pi nx)(1-x)dx\right|\\
	&=\left|\frac{\pi |n|}{\frac{a^2}{4}+\pi^2n^2}-\frac{a\pi |n|\left(e^{a/2}(-1)^{|n|}-1\right)}{\left(\frac{a^2}{4}+\pi^2n^2\right)^2}\right|
	\geq \frac{c}{|\l_n|}
      \end{split}
    \end{equation*}
    for some constant $c>0$ and for all sufficiently large $|n|$.
    By~\cite[Thm. 6.3]{Pau17b} we have
    $\|R(is,A)\|=O( s^2)$ for $|s|$ large, and thus~\cite[Thm. 2.4]{BorTom10} implies that the semigroup generated by $A=A_0-DD^\ast$ is polynomially stable with $\a=2$.
  \end{proof}

\begin{remark}
Note that if in this weakly damped Webster's equation one takes $a=0$ then we get a weakly damped wave equation on the interval $(0,1)$ with the same damping coefficient $d(x)=1-x$ and such equation is also polynomially stable with $\a= 2$.
\end{remark}

We consider the weakly damped Webster's equation with additional perturbing terms of the form
\begin{equation}\label{pweb_weak}
    \begin{split}
      &w_{tt}(t,x) - w_{xx}(t,x)-aw_x(t,x)+d(x)\int^1_0w_t(t,\xi) d(\xi)d\xi\\
      &=b_2(x)\int^1_0\left( w(t,\xi) c_1(\xi)+w_t(t,\xi)c_2(\xi)\right)e^{a\xi}d\xi,
      \end{split}
    \end{equation}
where $b_2,c_2\in L^2(0,1)$ and $c_1\in H_{-1/2}(L)$.
The following theorem presents conditions for the polynomial stability of the perturbed Webster's equation~\eqref{pweb_weak}. 
The spaces $H_{\theta}(L)$ and the corresponding norms are defined as in~\eqref{E:Hspace}.
The above perturbations correspond to rank one perturbation operators in the abstract wave equation. Addition of multiple perturbation terms can be treated similarly as in the case of the two-dimensional wave equation in Section~\ref{sec:2Dwave}.

\begin{theorem}\label{T:web_pert}
  Assume that the weakly damped Webster's equation \eqref{web_weak} is polynomially stable with $\alpha\leq2$, that $0\leq\beta,\gamma\leq 1$ such that $\beta+\gamma\geq\alpha$, and that $\kappa>0$ is as in Theorem~\textup{\ref{T:1.1}}.
    If
    $ b_2\in H_{\beta/2}(L)$, $c_1\in H_{{(\gamma-1)/2}}(L)$, $c_2\in H_{\gamma/2}(L)$
satisfy
   \begin{equation*}
    \left\|b_2\right\|_{H_{\beta/2}}<\frac{\kappa}{K_\beta},\quad \left\|c_1\right\|^2_{H_{(\gamma-1)/2}}+\left\|c_2\right\|^2_{H_{\gamma/2}}<\frac{\kappa^2}{K_\gamma^2},
 \end{equation*}
  then the perturbed Webster's equation~\eqref{pweb_weak} is polynomially stable with the same $\alpha$.
  Here $K_\t = e^{\frac{1}{2}{\pi^2\t(1-\t)}} M^\t$, $\t\in[0,1]$, and 
$M=1+ \|d\|_{L_a^2}(\frac{a^2}{4}+\pi^2)^{-1/2}.$
    \end{theorem}
\begin{proof}
  The perturbed system is of the form~\eqref{PGWE}  with
\begin{align*}
  B_2 = b_2\in L^2(\Omega), \qquad C_1 = \langle\cdot,c_1\rangle_{H_{1/2},H_{-1/2}}, \qquad \mbox{and} \qquad C_2 = \langle\cdot,c_2\rangle_{L^2},
\end{align*}
where $\langle\cdot,\cdot\rangle_{H_{1/2},H_{-1/2}}$ denotes the dual pairing between
$H_{1/2}(L)$
and $H_{-1/2}(L)$.
We have $B_2=b_2$, $C_1^\ast = c_1$ and $C_2^\ast = c_2$, and
 $\norm{D_0}=\norm{d}_{L_a^2}$. Since $-L$ is positive and its smallest eigenvalue is given by $-\mu_1 = a^2/4+\pi^2$, we also have
 $\norm{(-L)^{-1/2}}=(-\mu_1)^{-1/2} = (a^2/4+\pi^2)^{-1/2}$.
Thus the claim follows from Corollary~\ref{C:mainth}.
\end{proof}

As shown in Proposition~\ref{CorWeb} the Webster's equation with the damping function $d(x)=1-x$ is polynomially stable with $\a=2$. 
Since $\norm{d}_{L_a^2}^2=2a^{-3}(e^a-1-a-a^2/2)$ for this $d$, for the choices $\beta=\gamma=1$ Theorem~\ref{T:web_pert} has the following form.

\begin{corollary}\label{webs11}
  Let $d(x)=1-x$.
  If $\kappa>0$ is as in Theorem~\textup{\ref{T:1.1}} with $\beta=\gamma=1$
and if
    $ b_2,c_2\in H_{1/2}(L)$, $c_1\in L_a^2(0,1)$
satisfy
   \begin{equation*}
    \left\|b_2\right\|_{H_{1/2}}<\frac{\kappa}{M},\quad \left\|c_1\right\|^2_{L_a^2}+\left\|c_2\right\|^2_{H_{1/2}}<\frac{\kappa^2}{M^2},
 \end{equation*}
where  
$M=1+ 2a^{-3}(e^a-1-a-a^2/2)(\frac{a^2}{4}+\pi^2)^{-1/2},$
  then the perturbed Webster's equation~\eqref{pweb_weak} is polynomially stable with $\alpha=2$.
  
    \end{corollary}
\begin{example}
We use Theorem~\ref{thm:kappaestimate} for computing an explicit numerical value of $\kappa$ for the case $d(x) = 1-x$ and $a = 2$. 
To this end, we need to find functions $\eta(\cdot)$ and $\delta(\cdot)$ such that the condition~\eqref{eq:WPestimate} in Theorem~\ref{thm:kappaestimate} is satisfied.
For $a=2$ the eigenvalues and the corresponding eigenvectors of $A_0$ are $\l_{n}=\sign(n) i\sqrt{1+\pi^2n^2}$, 
		$$
		\psi_n(x)=\frac{1}{\l_n}\begin{pmatrix}
  \varphi_{|n|}(x)\\
  \l_n \varphi_{|n|}(x)
\end{pmatrix}
  , \quad \text{where }\varphi_{|n|}(x)=e^{-x}\sin(\pi |n|x), ~n\in\mathbb Z\backslash\{0\}.
		$$
		For all $n\in \mathbb{N}$ (using the inequality $\sqrt{x + y}\leq \sqrt{x}+ \sqrt{y}$)
		\begin{equation*}
		\begin{split}
			\dist(\l_n,\l_{n+1}) &= \dist(\l_{-n},\l_{-(n+1)}) = \sqrt{1+\pi^2(n+1)^2}-\sqrt{1+\pi^2n^2}\\
			&= \frac{\pi^2(2n+1)}{\sqrt{1+\pi^2(n+1)^2}+\sqrt{1+\pi^2n^2}} \geq  \frac{\pi^2(2n+1)}{2 + \pi(2n+1)} \geq \frac{3\pi^2}{2 + 3\pi},
		\end{split}
	\end{equation*}
	since $f(x) = \frac{\pi^2x}{2 + \pi x}$ is increasing for $x\in (1,\infty)$.
	If we choose $\delta(s) \equiv \delta_0 = \frac{\pi^2}{a + 3\pi}$, then every interval $(i(s-\delta_0),i(s+\delta_0))$ contains at most one eigenvalue and $\ran P_{(s-\delta_0,s+\delta_0)}$ consists of the corresponding eigenvector. Similar computations as in the proof of Proposition~\ref{CorWeb} then show that
		\begin{equation*}
      \begin{split}
	|D^*\psi_n|&=\left|\langle\varphi_{|n|},
	d\rangle_{L_a^2}\right|
	=\frac{\pi |n|}{1+\pi^2n^2}\left(1-\frac{2\left(e(-1)^{|n|}-1\right)}{\left(1+\pi^2n^2\right)^2}\right)\\
	&\geq \frac{c}{|\l_n|}\geq \frac{c}{s+\delta_0} = \eta(s)\|\psi_n\|,
      \end{split}
    \end{equation*}
	where $\eta(s) = \frac{c}{s+\delta_0} $ with the $c>0$ such that
		$$
		c \leq \inf\left\{ \frac{\pi |n|}{\sqrt{1+\pi^2n^2}}\left(1-\frac{2\left(e(-1)^{|n|}-1\right)}{\left(1+\pi^2n^2\right)^2}\right)\right\}.
		$$
		To find a suitable $c>0$, let us denote
		$$
		F(n) = \frac{\pi |n|}{\sqrt{1+\pi^2n^2}}\left(1-\frac{2\left(e(-1)^{|n|}-1\right)}{\left(1+\pi^2n^2\right)^2}\right)
		$$
		and 
		$$
		G(n) = \frac{\pi |n|}{\sqrt{1+\pi^2n^2}}\left(1-\frac{2\left(e-1\right)}{\left(1+\pi^2n^2\right)^2}\right).
		$$
		It is obvious that $F(n)\geq G(n)$ for $n\geq 2$. The values $G(n)$ for $n\geq 2$ are increasing and $G(n)\to 1$ as $n\to \infty$ and therefore $\min\limits_{n\geq 2} G(n) = G(2)$.
		Hence we can choose
		$$
		c = \min\{F(1), G(2)\} = \frac{2\pi }{\sqrt{1+4\pi^2}}\left(1-\frac{2\left(e-1\right)}{\left(1+4\pi^2\right)^2}\right).
		$$
		Finally, the maximum of $\eta(s)$ when $s\geq 0$ is $\eta_0 = c/\delta_0$.
		
		In the next step we calculate $M_R$. To this end, we need also $\|D\|$ which is 
		$$
		\|D\| = \norm{d}_{L_2^2} = \frac{\sqrt{e^2-5}}{2}.
		$$
		We can now use Matlab to compute $M_R = 5.451$.
		
		Now we will find the constant $M_0>0$. A direct estimate using $(x+y)^2\leq 2(x^2+y^2)$ and $\delta_0<1$ shows that
		$$
		\eta(s)^{-2}\delta(s)^{-2} = \frac{(s+\delta_0)^2}{c^2\delta_0^2}\leq \frac{2(s^2+\delta_0^2)}{c^2\delta_0^2}\leq \frac{2}{c^2\delta_0^2}(s^2+1) = M_0(s^2+1),
		$$
		with $M_0=\frac{2}{c^2\delta_0^2}$. 
		To compute $M_C>0$, we also need an estimate for $\|A^{-1}\|$. We have 
		\begin{equation*}
		\begin{split}
    \|A^{-1}\|&=\left\|(A_0-DD^*)^{-1}\right\|\leq\left\|(I-DD^*A_0^{-1})^{-1}\right\|\|A_0^{-1}\|\\
    &=\left\|I+DD^*A_0^{-1}\right\|\|A_0^{-1}\|\leq \left(1+\norm{d}_{L_2^2}^2 \norm{(-L)^{-1/2}}\right)\norm{(-L)^{-1/2}}\\ 
		&= \left(1 + \frac{e^2-5}{4\sqrt{1+\pi^2}}\right)\frac{1}{\sqrt{1+\pi^2}}.
  \end{split}
 \end{equation*}
 If we take $s_0 = 2.8$ in the formula for $M_C$, we obtain $M_C = 17.0664$. This way, we finally see that $\kappa>0$ in Theorem~\ref{T:1.1} can take any value such that   
 $\kappa < \frac{1}{\sqrt{2M_C}}= 0.1712$. 

Now we are able to give explicit upper bounds for the norms of $b_2$, $c_1$, and $c_2$ for the preserving of polynomial stability. From Corollary \ref{webs11} we have that $ b_2,c_2\in H_{1/2}(L)$, $c_1\in L_a^2(0,1)$
satisfy
   \begin{equation*}
    \left\|b_2\right\|_{H_{1/2}}<\frac{\kappa}{M} = 0.1449,\quad \sqrt{\left\|c_1\right\|^2_{L_a^2}+\left\|c_2\right\|^2_{H_{1/2}}}<\frac{\kappa}{M}= 0.1449,
 \end{equation*}
where $M = 1 + \frac{e^2-5}{4\sqrt{1+\pi^2}}$, then the perturbed Webster's equation~\eqref{pweb_weak} is polynomially stable with $\alpha=2$.
\end{example}

\section{Wave equation with an acoustic boundary condition}\label{S:ABC}

In this section we consider a one-dimensional wave equation with an "acoustic boundary condition" on the interval $\Omega=(0,1)$,
\begin{equation}\label{ABcon}
  \left\{
  \begin{array}{ll}
    w_{tt}(t,x)
    =w_{xx}(t,x) \text{ in } (0,\infty)\times \Omega    \\
    a_{tt}(t)=-ka(t)-da_t(t)-w_t(1,t)\\
    w_x(t,1)=a_t(t),\quad w_x(t,0)=0,\\
    w(0,x)=w_0(x),\quad w_t(0,x)=w_1(x), \quad a(0)=a_0, \quad a_t(0)=a_1
  \end{array}\right.
\end{equation}
with $k,d>0$~\cite[Sec.~6.1]{AbbNic13}.
The spectral properties and polynomial stability of differential equations of this form (also on multidimensional spatial domains) have been studied in detail in~\cite{Bea76,RivQin03,AbbNic13,AbbNic15b}. In particular, 
it was shown in \cite[Thm.~1.3]{RivQin03} 
 that the energy of the classical solutions of~\eqref{ABcon} decays at a rational rate, and the optimality of this decay rate was proved in~\cite[Sec.~6.1]{AbbNic13}.
This model is not of the form~\eqref{GWE}, but the preservation of its polynomial stability can be studied using Theorem~\ref{T:1.1}.

Equation~\eqref{ABcon} can be formulated as an abstract Cauchy problem 
with state $u(t) = ( w_x(t,\cdot), w_t(t,\cdot), a(t), a_t(t))^\top$ 
on the Hilbert space 
$
\cH=L^2(0,1)\times L^2(0,1)\times \mathbb{C}^2
$
 with inner product defined as
$$
\langle u,v\rangle_\cH=\left\langle u_1,v_1\right\rangle_{L^2}+\left\langle u_2,v_2\right\rangle_{L^2}+ku_3\ov{v_3}+u_4 \ov{v_4}
$$
for all $u=(u_1(\cdot),u_2(\cdot),u_3,u_4)^\top$, $v=(v_1(\cdot),v_2(\cdot),v_3,v_4)^\top\in\cH$.
In this situation the semigroup generator is defined as
\begin{equation*}
  A=
  \begin{pmatrix}
    0&\partial _x&0&0\\
    \partial_x&0&0&0\\
    0&0&0&1\\
    -C_0&0&-k&-d
  \end{pmatrix}
  , \qquad
  C_0f=f(1) \quad \mbox{for} ~ f\in H^1(0,1), 
\end{equation*}
with domain
$$
\dom A=\left\{
(
  u_1(\cdot),
  u_2(\cdot),
  u_3,
  u_4
)^\top
\in (H^1(0,1))^2 \times\mathbb C^2:u_2(0)=0,u_2(1)=u_4
\right\}.
$$
The operator $A$ generates a contraction semigroup on $\cH$, and it was shown in \cite[Thm.~1.3]{RivQin03} 
(see also~\cite[Sec.~6.1]{AbbNic13},~\cite[Sec.~4]{PauSOTA18}) that this semigroup is polynomially stable with $\alpha=2$.
In the context of the wave equation~\eqref{ABcon} this means that there exists a constant $\Mpert>0$ such that for all initial conditions $w_0,w_1,a_0,a_1$ such that $(w_0',w_1,a_0,a_1)^\top\in \Dom A $
the solutions of~\eqref{ABcon} satisfy
\begin{equation}
  \label{ABcondecay}
  \begin{split}
    &\norm{w_x(t,\cdot)}_{L^2}^2 +  
    \norm{w_t(t,\cdot)}_{L^2}^2 +  
     |a(t)|^2 + |a_t(t)|^2 \\
    &\hspace{4cm}\leq \frac{\Mpert}{t} \left( \norm{w_0''}_{L^2}^2 + \norm{w_1'}_{L^2} + k |a_0|^2 + |a_1|^2 \right)
  \end{split}
\end{equation}
  for all $t>0$.

We can now study the stability of perturbed wave equations of the form
\begin{equation}\label{ABconPer}
  \left\{
  \begin{array}{ll}
    w_{tt}(x,t)
    =w_{xx}(x,t)+b_2(x)\int_0^1 (w_t(\xi,t)\ov{c_1(\xi)}+ w_x(\xi,t)\ov{c_2(\xi)})d\xi \\[1ex]
      \hspace{1.8cm}+b_2(x)(ka(t)\ov{c_3}+a_t(t)\ov{c_4})\\[1ex]
    a_{tt}(t)=-ka(t)-da_t(t)-w_t(1,t)\\
    w_x(t,1)=a_t(t),\quad w_x(t,0)=0,\\
    w(0,x)=w_0(x),\quad w_t(0,x)=w_1(x), \quad a(0)=a_0, \quad a_t(0)=a_1
  \end{array}\right.
\end{equation}
where $b_2,c_1,c_2\in L^2(0,1)$ and $c_3,c_4\in\dC$.
The following two theorems introduce conditions for the polynomial stability of~\eqref{ABconPer}.
\begin{theorem}
  \label{T:ABcon1}
Assume $\kappa>0$ is as in Theorem~\textup{\ref{T:1.1}} with $\beta=\gamma=1$.
If $b_2\in H^1_0(0,1)$, $c_1,c_2\in H^1(0,1)$, and $c_4=0$ satisfy
   \begin{align*}
     \|b_2'\|_{L^2}<\kappa,\quad \mbox{and} \quad
4\norm{c_1}_{H^1}^2 + \norm{c_2'}_{L^2}^2 +  3k^2|c_3|^2<\kappa^2,
   \end{align*}
   then the perturbed equation~\eqref{ABconPer} is polynomially stable with $\a=2$.
For such perturbations there exists $\Mpert>0$ such that the solutions of~\eqref{ABconPer} corresponding to 
initial conditions $w_0,w_1,a_0,a_1$ such that $(w_0',w_1,a_0,a_1)^\top\in \Dom A $
 satisfy~\eqref{ABcondecay} for all $t>0$.
  \end{theorem}
  \begin{proof}
 The perturbed system operator can be written as $A+BC$ where
 $B=(0,b_2(\cdot),0,0)^\top\in \Lin(\dC,\cH)$ and $C=(\langle\cdot,c_1(\cdot)\rangle_{L^2},\langle\cdot,c_2(\cdot)\rangle_{L^2},\ov{c_3},0)\in \Lin(\cH,\dC)$ with $b_2\in H_0^1(0,1)$, $c_1,c_2\in H^1(0,1)$ and $c_3 \in \mathbb C$.
A straightforward computation shows that the adjoint operator of $A$ has the form
\begin{equation*}
  A^*=
  \begin{pmatrix}
    0&-\partial _x&0&0\\
    -\partial_x&0&0&0\\
    0&0&0&-1\\
    C_0&0&k&-d
  \end{pmatrix}
\end{equation*}
and that its domain $\Dom A^\ast$ contains the subspace
$$
\left\{
 (u_1(\cdot),
  u_2(\cdot),
  u_3,
  u_4)^\top
\in H^1\times H^1\times \mathbb C\times \mathbb C
:u_2(1)=u_4, u_2(0)=0\right\}.
$$
The assumptions therefore imply that $\ran B\subset \Dom A$ and $\ran C^\ast \subset \dom A^\ast$, and 
\begin{align*}
  \norm{AB}&=\norm{b_2'}_{L^2},  \\
  \norm{A^\ast C^\ast}^2 
  &= \norm{c_1'}_{L^2}^2 + \norm{c_2'}_{L^2}^2 + |c_1(1)+kc_3|^2\\
  &\leq \norm{c_1}_{L^2}^2 + \norm{c_2'}_{L^2}^2 + \left(\norm{c_1}_{L^2}+\norm{c_1'}_{L^2}+ k|c_3|\right)^2\\
  &\leq 4\norm{c_1}_{H^1}^2 + \norm{c_2'}_{L^2}^2 +  3k^2|c_3|^2.
\end{align*}
Here we have used the property $|c_1(1)|\leq \norm{c_1}_{L^2}+\norm{c_1'}_{L^2}$, which can be verified using the identity $c_1(1)=\int_0^1 \frac{d}{dx}(xc_1(x))dx = \int_0^1(c_1(x)+xc_1'(x))dx$.
Thus the claim follows from Theorem~\ref{T:1.1} with $\beta=\gamma=1$.
  \end{proof}

   Similarly, applying Theorem \ref{T:1.1} with $\beta=2$ and $\gamma=0$ we obtain the following alternative conditions for the polynomial stability of~\eqref{ABconPer}.

 \begin{theorem}
Assume $\kappa>0$ is as in Theorem~\textup{\ref{T:1.1}} with $\beta=2$ and $\gamma=0$.
If $b_2\in H_0^1(0,1)\cap H^2(0,1)$, $c_1,c_2\in L^2(0,1)$, and $c_3,c_4\in\dC$ satisfy
\begin{align*}
     \|b_2'\|_{H^1}<\kappa,
     \qquad 
     \|c_1\|^2_{L^2}+\|c_2\|^2_{L^2}+k|c_3|^2+|c_4|^2<\kappa^2,
   \end{align*}
   then \eqref{ABconPer} is polynomially stable with $\a=2$.
For such perturbations there exists $\Mpert>0$ such that the solutions of~\eqref{ABconPer} corresponding to 
initial conditions $w_0,w_1,a_0,a_1$ such that $(w_0',w_1,a_0,a_1)^\top\in \Dom A $
 satisfy~\eqref{ABcondecay} for all $t>0$.  
\end{theorem}
  \begin{proof}
    The perturbations have the same form as in the proof of Theorem~\ref{T:ABcon1}. 
    Since $B=b:=(0,b_2(\cdot),0,0)^\top$ with $b_2\in H_0^1(0,1)\cap H^2(0,1)$, we have
     $(0,b_2(\cdot),0,0)^\top\in\dom A$ and
$Ab=(b'_2(\cdot),0,0,0)^\top\in\dom A$.
Thus $\ran B\subset \dom A^2$ and
     \begin{align*}
       A^2 \pmat{0\\b_2\\0\\0}=A \pmat{b_2'\\0\\0\\0} = \pmat{0\\b_2''\\0\\-b_2'(1)}.
    \end{align*} 
    implies $\norm{A^2B}^2=\norm{b_2''}_{L^2}^2+|b_2'(1)|^2\leq 3\norm{b_2'}_{H^1}^2$, since $|b_2'(1)|^2\leq 2\norm{b_2'}_{H^1}^2$ similarly as in the proof of Theorem~\ref{T:ABcon1}.
The claim now follows from Theorem~\ref{T:1.1} with the choices $\beta=2$ and $\gamma=0$.  
  \end{proof}

\end{document}